\documentclass[authoryear]{elsarticle}
\usepackage{amsfonts}
\usepackage{amsmath}
\usepackage[english]{babel}
\usepackage{inputenc}
\usepackage{natbib}
\usepackage{geometry}
\usepackage{float}
\usepackage{graphicx}
\usepackage{caption}
\usepackage{subcaption}
\usepackage{amsthm}

\restylefloat{table}

\def\S{\mathbb{S}}

\newtheorem{prop}{Proposition}

\usepackage{color}

\begin{document}
\begin{frontmatter}

\title{Log Gaussian Cox processes on the sphere}


\author[label1]{Francisco Cuevas-Pacheco}
\ead{francisco@math.aau.dk}

\author[label1]{Jesper M{\o}ller\corref{cor1}}
\address[label1]{Department of Mathematical Sciences, Aalborg University, Denmark}
\ead{jm@math.aau.dk}

\cortext[cor1]{Corresponding author}

\begin{abstract}
	\noindent A log Gaussian Cox process (LGCP) is a doubly stochastic construction consisting of a Poisson point process with a random log-intensity given by a Gaussian random field. Statistical methodology have mainly been developed for LGCPs defined in the $d$-dimensional Euclidean space. This paper concerns the case of LGCPs on the $d$-dimensional sphere, with $d=2$ of primary interest. We discuss the existence problem of such LGCPs, provide sufficient existence conditions, and establish further
	useful theoretical properties. The results are applied for the description of sky positions of galaxies, in comparison with previous analysis based on a Thomas process, using simple estimation procedures and making a careful model checking. We account for inhomogeneity in our models, and as the model checking is based on a thinning procedure which produces homogeneous/isotropic LGCPs, we discuss its sensitivity.
\end{abstract}


\end{frontmatter}


\section{Introduction}

Statistical analysis of point patterns on the sphere have been of interest for a long time, see \cite{lawrence2016point},  \cite{Moller_summary_2017}, and the references therein. 
Although models and methods developed for planar and spatial point processes may be adapted, statistical methodology for point processes on the sphere is still not so developed as discussed in the mentioned references and in the following. 

The focus has been on developing functional summary statistics, parametric models, and inference procedures. For homogeneous point patterns, \citet{Robeson2014} studied Ripley's $K$-function on the sphere \citep{Ripley76,Ripley77}, and \citet{lawrence2016point} provided a  careful presentation of Ripley's $K$-function and other functional summary statistics such as the empty space function $F$, the nearest-neighbour distance function $G$, and the $J=(1-G)/(1-F)$ function, including how to account for edge effects (i.e.\ when the data is a point pattern observed within a window strictly included in the sphere and unobserved points outside this window may effect the data). See also \citet{Moller_summary_2017} for details and the connection to reduced Palm distributions. 
For inhomogeneous point patterns, \citet{lawrence2016point} and \citet{Moller_summary_2017} studied the pair correlation function
and the inhomogeneous $K$-function. The models which have been detailed are rather scarse: Homogeneous Poisson point process models \citep{Raskin94,Robeson2014}; determinantal point processes \citep{Moller_summary_2017,Moller2018} for regular/repulsive point patterns; inhomogeneous Poisson point process models and Thomas point process models for aggregated/clustered point patterns (Lawrence et al., 2016\nocite{lawrence2016point}; Section~\ref{s:PoissonCox}-\ref{s:data} in the present paper); and inhomogeneous log Gaussian Cox processes (LGCPs) for aggregated/clustered point patterns using the R-INLA approach (Simpson et al., 2016\nocite{Simpson.et.al}; Section~\ref{s:final2} in the present paper). 

This paper treats inhomogeneous aggregated/clustered point patterns on the sphere and studies how the theory of Cox processes and in particular LGCPs \citep{Moller1998} can be adapted to analysing such data, using other LGCPs models than in \citet{Simpson.et.al}
and 
simpler inference procedures than the R-INLA approach, where we exploit the nice moment properties of LGCPs. In particular,
we demonstrate that an inhomogeneous LGCP provides a better description of the sky positions of galaxies than analysed in \citet{lawrence2016point} by using an inhomogeneous Thomas process.
For comparison, as in  \citet{lawrence2016point}, we use a minimum contrast procedure for parameter estimation, where we discuss the  
sensibility of the choice of user-specified parameters. No model checking was done for the fitted inhomogeneous Thomas process in \citet{lawrence2016point}. 
Moreover, we show how a thinning procedure can be applied 
 to generate homogeneous point patterns so that the $F,G,J$-functions can be used for model checking, and we discuss the sensitivity of this thinning procedure.

The paper is organised as follows. Section~\ref{s:background} provides the setting and needed background material on point processes, particularly on Poisson and Cox processes, the data example of sky positions of galaxies, and the inhomogeneous Thomas process introduced in \citet{lawrence2016point}. Section~\ref{s:LGCP} contains the definition and existence conditions for LGCPs on the sphere, studies their useful properties, and compares the 
fitted Thomas processes and LGCPs for the data example. 
Finally, Section~\ref{s:final} summarizes our results, establishes some further results, and discusses future directions for research. 

\section{Background}\label{s:background}
\subsection{Setting}
Let $\mathbb{S}^{d} = \{ x \in \mathbb{R}^{d+1}: \|x\| = 1\}$ denote the $d$-dimensional unit sphere included in the $(d+1)$-dimensional Euclidean space $\mathbb{R}^{d+1}$, equipped with the usual inner product $\langle x, y\rangle=\sum_{i=0}^dx_iy_i$ for points $x=(x_0,\ldots,x_d),y=(y_0,\ldots,y_d)\in\mathbb R^{d+1}$  and the usual length $\|x\|=\sqrt{\langle x, x\rangle}$. We are mainly interested in the case of $d=2$. By a region $U\subseteq \mathbb{S}^{d}$ we mean that $U$ is a Borel set. 
For unit vectors $u,v\in\mathbb{S}^{d}$, let  $d(u,v) = \arccos(\langle u,v \rangle)$
be the geodesic distance on the sphere. 

By a point process on $\mathbb{S}^{d}$, we understand a random finite subset $\textbf{X}$ of $\mathbb{S}^{d}$. 
We say that $\textbf{X}$ is isotropic if $O\textbf{X}$ is distributed as $\textbf{X}$ for any $(d+1)\times(d+1)$ rotation matrix $O$. We assume that $\textbf{X}$ has an intensity function, $\lambda(u)$, and a pair correlation function, $g(u,v)$, meaning that if $U,V\subseteq\mathbb S^d$ are disjoint regions and $N(U)$ denotes the cardinality of $X\cap U$, then
\[\mathrm E[N(U)]=\int_U\lambda(u)\,\mathrm du,\qquad 
\mathrm E[N(U)N(V)]=\int_U\int_V\lambda(u)\lambda(v)g(u,v)\,\mathrm du\,\mathrm dv,\]
where $\mathrm du$ is the Lebesgue/surface measure on $\mathbb{S}^{d}$. We say that $\textbf{X}$ is (first order) homogeneous if $\lambda(u)=\lambda$ is constant, and second order intensity reweighted homogeneous if $g(u,v)=g(r)$ only depends on $r=d(u,v)$. Note that these properties are implied by isotropy of $\textbf{X}$, and second order intensity reweighted homogeneity allows to define the (inhomogeneous) $K$-function \citep{Baddeley2000} by
\[K(r)=\int_{d(u,v)\le r} g(u,v)\,\mathrm dv,\qquad 0\le r\le\pi,\]
for an arbitrary chosen point $u\in\mathbb S^d$, where $\mathrm dv$ denotes Lebesgue/surface measure on $\mathbb S^d$. For instance,
\begin{equation}\label{e:K-g}
K(r)=2\pi\int_0^r g(s)\sin s\,\mathrm ds\qquad \mbox{if }d=2.
\end{equation}

\subsection{Poisson and Cox processes}\label{s:PoissonCox}

For a Poisson process $\mathbf X$ on $\S^{d}$ with intensity function $\lambda$, $N(\mathbb S^d)$ is Poisson distributed with mean $\int\lambda(u)\,\mathrm du$, and conditional on $N(\mathbb S^d)$, the points in $\mathbf X$ are independent identically distributed with a density proportional to $\lambda$. The process is second order intensity reweighted homogeneous, with $K$-function
\begin{equation}\label{e:KPois}K_{\mathrm{Pois}}(r)=2\pi(1-\cos r)\qquad\mbox{if }d=2.\end{equation}

Let $\mathbf{Z}=\{\mathbf{Z}(u):u\in\mathbb S^d\}$ be a non-negative random field so that almost surely $\int\mathbf{Z}(u)\,\mathrm du$ is well-defined and finite, and 
$Z(u)$ has finite variance for all $u \in\mathbb{S}^d$. We say that $\mathbf{X}$ is a Cox process driven by $\mathbf{Z}$ if $\mathbf{X}$ conditional on $\mathbf{Z}$ is a Poisson process on $\S^{d}$ with intensity function $\mathbf{Z}$. Then $\mathbf{X}$
has intensity function 
\begin{align}\label{cox_intensity1}
\lambda(u) = \mathbb{E}[\mathbf{Z}(u)],
\end{align}
and defining the residual driving random field by
$\mathbf Z_0=\{\mathbf Z_0(u):u\in\mathbb S^d\}$, where $\mathbf Z_0(u)=\mathbf Z(u)/\lambda(u)$ (setting $0/0=0$), $\mathbf{X}$
has pair correlation function 
\begin{align}\label{cox_intensity2}
g(u,v) = {\mathbb{E}\left[\mathbf{Z_0}(u)\mathbf{Z_0}(v)\right]}. 
\end{align}
Thus
$\mathbf{X}$ is second order intensity reweighted homogeneous if $\mathbf Z_0$ is isotropic, that is, $\{\mathbf Z_0(Ou):u\in\mathbb S^d\}$ is distributed as $\mathbf Z_0$ for any $(d+1)\times(d+1)$ rotation matrix $O$. 
We shall naturally specify such Cox processes in terms of $\lambda$ and $\mathbf Z_0$.

\subsection{Data example and inhomogeneous Thomas process}\label{s:data}

Figure~\ref{galaxies_plot_a} shows the sky positions of 10,546 galaxies from the Revised New General Catalogue and Index Catalogue (RNGC/IC) \citep{steinicke2015}.
Here, we are following \citet{lawrence2016point} in making a rotation of the original data so that the two circles limit a band around the equator, which is an approximation of the part of the sky obscured by the Milky Way, and the observation window $W\subset\mathbb S^2$ is the complement of the band; 64 galaxies contained in the band are omitted. 

\begin{figure} 
\begin{subfigure}{.5\linewidth}
\centering
\includegraphics[scale=0.4]{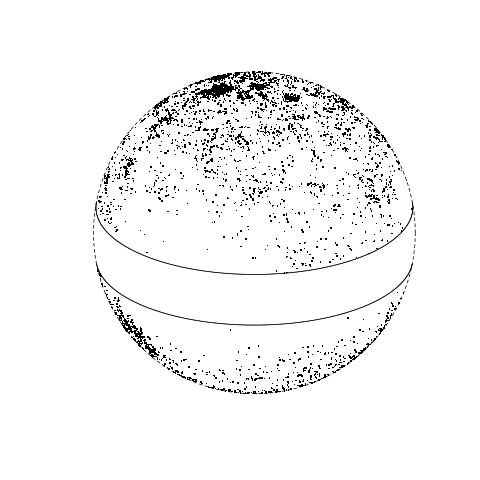} 
\caption{Sky positions of galaxies before thinning.}
\label{galaxies_plot_a}
\end{subfigure}
\begin{subfigure}{.5\linewidth}
\centering
\includegraphics[scale=0.4]{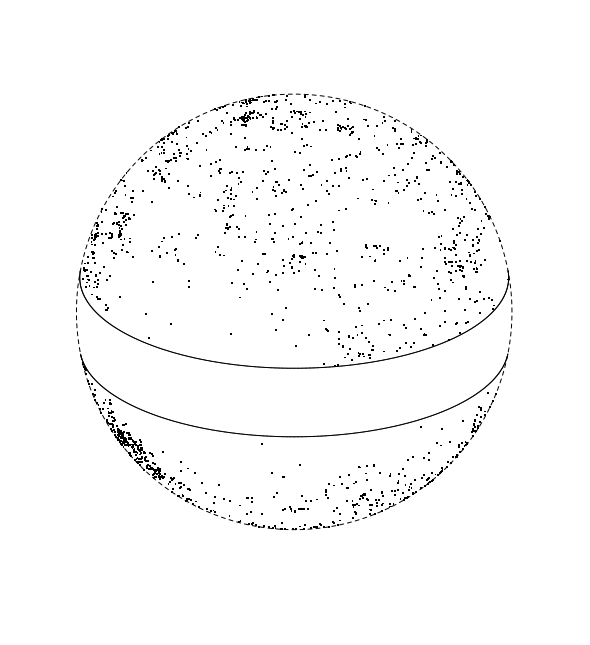} 
\caption{Sky positions of galaxies after thinning.}\label{galaxies_plot_b}
\end{subfigure}
\caption{The sky positions of (a) 10,546 galaxies and (b) 3,285 galaxies obtained after a thinning procedure so that a homogeneous point pattern is expected to be obtained. The circles limit the part of the sky obscured by the Milky Way.}\label{galaxies_plot}
\end{figure}

Different plots and tests in the accompanying supporting information to \citet{lawrence2016point} clearly show that the galaxies are aggregated and not well-described by an inhomogeneous Poisson process model. \citet{lawrence2016point} fitted the intensity function
\begin{align}
\lambda(u) =  
6.06 -0.112 \sin\theta \cos\phi -0.149 \sin\theta \sin\phi + 0.320\cos\theta + 1.971 \cos^2\theta, 
\label{intensity}
\end{align}
where $u=(\sin\theta\cos\phi,\sin\theta\sin\phi,\cos\theta)$, $\theta\in[0,\pi]$ is the colatitude, and $\phi\in[0,2\pi)$ is the longitude. The term $-0.112 \sin\theta \cos\phi -0.149 \sin\theta \sin\phi + 0.320\cos\theta$ is the inner product of $u$ with $(-0.112,-0.149,0.32)$, and the term $1.971 \cos^2\theta$ allows a gradient in intensity from the poles to the equator, cf.\ the discussion in  \citet{lawrence2016point} and their plot Movie 2 (where we recommend using the electronic version of the paper and the link in the caption to this plot).
Throughout this paper, we also use \eqref{intensity}. 

\citet{lawrence2016point} proposed an inhomogeneous Thomas process, that is, a Cox process with intensity function \eqref{intensity} and isotropic
residual driving random field given by
\begin{align}\mathbf Z_0(u)=\sum_{y\in\mathbf Y}f_{y,\xi}(u)/\kappa,\qquad u\in\mathbb S^2,\label{e:res}
\end{align}
where  $\mathbf Y$ is a homogeneous Poisson process on $\mathbb S^2$ with intensity $\kappa>0$, and where 
\[f_{y,\xi}(u)=\frac{\xi}{4\pi\sinh\xi}\exp\left(\xi\langle u,y\rangle\right),\qquad u\in\mathbb S^2,\]
is the density for a von Mises-Fisher distribution on $\mathbb S^2$ with mean direction $y$ and concentration parameter $\xi>0$. They estimated $\kappa$ and $\eta$ by a minimum contrast procedure \citep{Diggle1983,diggle2013statistical}, where a non-parametric estimate $\widehat K$ is compared to 
\[K(r)=K_{(\kappa,\eta)}(r)=K_{\mathrm{Pois}}(r)+\frac{\cosh(2\xi)-\cosh\left(\sqrt{2\xi^2(1+\cos r)}\right)}{4\kappa\sin^2\xi}.\]  
Specifically, the minimum contrast estimate $(\hat\kappa,\hat\eta)$ is  minimising the contrast
\begin{equation}\label{e:contrast}
\int_a^b\left(\widehat K(r)^{0.25}-K_{(\kappa,\eta)}(r)^{0.25}\right)^2\,\mathrm dr,
\end{equation}
where $b>a\ge0$ are user-specified parameters. 
\citet{lawrence2016point} used the integration interval $[a,b]=[0,1.396]$, where 1.396 is the maximal value of the smallest distance from the north pole respective south pole to the boundary of $W$. They obtained $\hat\kappa=5.64$ and $\hat\xi=266.6$. 

\section{Log Gaussian Cox processes}\label{s:LGCP}
\subsection{Definition and existence}\label{s:DefExLGCP}

Let $\mathbf{Y}=\{\mathbf{Y}(u):u\in\mathbb S^d\}$ be a Gaussian random field (GRF), that is, $\sum_{i=1}^n a_i\mathbf Y(u_i)$ is normally distributed for any integer $n>0$, numbers $a_1,\ldots,a_n$, and $u_1,\ldots,u_n\in\mathbb S^d$. Let $\mu(u)=\mathrm E\mathbf Y(u)$ be its mean function and $c(u,v)=\mathrm E[(\mathbf Y(u)-\mu(u))(\mathbf Y(v)-\mu(v))]$ its covariance function. Assuming that almost surely $\mathbf Z:=\exp(\mathbf Y)$ is integrable, the Cox process $\mathbf X$ driven by $\mathbf{Z}$ is called a log Gaussian Cox process \citep[LGCP;][]{Moller1998,MW04}. 

Define the corresponding mean-zero GRF, $\mathbf Y_0:=\mathbf{Y}-\mu$, which has also covariance function $c$. 
Assuming $\mu$ is a Borel function with an upper bounded, then almost sure continuity of $\mathbf Z_0$ implies almost sure integrability of $\mathbf Z=\mathbf Z_0\exp(\mu)$, and so $\mathbf X$ is well-defined. 
In turn, almost sure continuity of $\mathbf Z_0$ is implied if almost surely $\mathbf Y_0$ is locally sample H\"{o}lder continuous of some order $k>0$, which means the following: With probability one, for every $s \in \S^{d}$, there is a neighbourhood $V$ of $s$ and a constant $C_{V,k}$ so that 
\begin{align}\label{hc_def}
\sup_{u,v \in V, u\neq v}  \left| \frac{\textbf{Y}_0(u) - \textbf{Y}_0(v)}{d(u,v)^{k}} \right| \leq C_{V,k}.
\end{align}     
The following proposition is a special case of  \citet[][Corollary~4.5]{lang2016continuity}, and it
 provides a simple condition ensuring \eqref{hc_def}.
\begin{prop}\label{p:1}
Let 
\[\gamma(u,v) = \mathbb{E}[(\textbf{Y}_0(u)-\textbf{Y}_0(v))^{2}]/2\]
 be the variogram of a mean-zero GRF $\mathbf Y_0$. Suppose there exist numbers $s \in (0,1],$ $\ell \in (0,1),$ and $m>0$ such that 
\begin{align}\label{condition_hc}
\gamma(u,v) \leq m \,d(u,v)^{\ell/2}\qquad\mbox{whenever }d(u,v) < s.
\end{align}
 Then, for any $k \in(0, \ell/2)$, 
 $\mathbf Y_0$ is almost surely locally sample H\"{o}lder continuous of order $k$.
\end{prop}


\subsection{Properties}

The following proposition is easily verified along similar lines as Proposition~5.4 in \citet{MW04} using \eqref{cox_intensity1}-\eqref{cox_intensity2} and the expression for the Laplace transform of a normally distributed random variable.

\begin{prop}\label{p:2} A LGCP $\mathbf X$ has intensity function and pair correlation functions given by
\begin{equation}\label{e:LGCP-intensities}
\lambda(u)=\exp(\mu(u)+c(u,u)/2),\qquad g(u,v)=\exp(c(u,v)),
\end{equation}
where $\mu$ and $c$ are the mean and covariance functions of the underlying GRF.
\end{prop}

In other words, the distribution of $\mathbf X$ is specified by $(\lambda,g)$ because 
$\mu(u)=\log\lambda(u)-\log g(u,u)/2$ and  $c(u,v)=\log g(u,v)$ 
specify the distribution of the GRF. 

By \eqref{e:LGCP-intensities}, second order intensity reweighted homogeneity of the LGCP is equivalent to isotropy of the covariance function, that is, $c(u,v)=c(r)$ depends only on $r=d(u,v)$. 
Parametric classes of isotropic covariance functions on $\S^{d}$ are provided by \citet{gneiting2013strictly}, see Table~\ref{tab:covariances} in Section~\ref{s:final1}, and by \citet{Moller2018}. In Section~\ref{s:LGCPdata}, we consider the so-called multiquadric covariance function which is isotropic and given by
\begin{align}\label{Multiquadric}
c_{(\sigma,\delta,\tau)}(r) = \sigma^{2}\left(\frac{(1-\delta)^{2}}{1+\delta^{2}-2\delta \cos r}\right)^{\tau}, \qquad r \in [0,\pi],
\end{align}
where $(\sigma,\delta,\tau)\in(0,\infty)\times(0,1)\times(0,\infty)$. 
\begin{prop}\label{p:3}
A mean-zero GRF $\mathbf Y_0$ with multiquadric covariance function is locally sample H\"{o}lder continuous of any order $k\in(0,1)$.
\end{prop}
\begin{proof}
By \eqref{Multiquadric}, for $r=d(u,v)$,
\begin{equation}\label{e:aa}
\gamma(u,v)=c_{(\sigma,\delta,\tau)}(0)-c_{(\sigma,\delta,\tau)}(r)=\sigma^2\left(1-c_{(1,\delta,1)}(r)^\tau\right).
\end{equation}
Here, letting
 $p = {2\delta}/{(1+\delta^{2})}$,  
 \begin{align}
c_{(1,\delta,1)}(r) &= \frac{(1-\delta)^{2}}{1+\delta^{2} - 2\delta \cos r} 
=\frac{1 - p}{1 - p \cos r } 
=1 + p\left(\frac{\cos r - 1}{1 - p \cos r}\right)\geq 1 + p(\cos r - 1)\nonumber \\
&\geq 1 - {p} r^2/{2},\label{star}
\end{align}
 where the first inequality follows from $0<p<1$ because $0<\delta<1$, and the second inequality follows from $1-\cos x\le x^2/2$ whenever $0\le x\le1$. If $\tau < 1$,  then 
 \[\gamma(u,v)\le \sigma^{2}\left(1 - c_{(1,\delta,1)}(r)\right)\le p\sigma^{2}r^{2}/2,\]
 where the first inequality follows from \eqref{e:aa} and the second from \eqref{star}.
If $\tau \geq 1$, then for any $\alpha\in(0,1/2)$,
\[\gamma(u,v)\le  \sigma^{2}\left(1 - \left(1 - {p}r^{2}/2\right)^{\tau}\right)\le 
p\sigma^{2}{\tau}r^{2}/2\leq p\sigma^{2}{\tau} r^{\alpha}/2,\]
where the first inequality follows from \eqref{e:aa}-\eqref{star} and the second from $1 - (1-x)^{\tau} \leq x\tau$ whenever $0\le x\le1$.  
Hence, for any $(\sigma,\delta,\tau)\in(0,\infty)\times(0,1)\times(0,\infty)$ and any $\alpha\in(0,1/2)$, \eqref{Multiquadric} satisfies \eqref{condition_hc} with $s=1$, $\ell = 2\alpha\in(0,1)$, and $m = p\sigma^{2}\tilde{\tau}/2$, where $\tilde{\tau} = \max\{1,\tau\}$.
\end{proof}

In fact locally sample H\"{o}lder continuity is satisfied when considering other commonly used parametric classes of covariance functions, 
but the proof will be case specific as shown in the proof of Proposition~\ref{p:gneitingclasses} in Section~\ref{s:final1}.

\subsection{Comparison of fitted Thomas processes and LGCPs for the data example}\label{s:LGCPdata}

In this section, to see how well the estimated Thomas process, $\mathbf X_{\mathrm{Thom}}$, obtained in \citet{lawrence2016point} fits the sky positions of galaxies discussed in Section~\ref{s:data}, we use methods not involving the $K$-function (or the related pair correlation function, cf.\ \eqref{e:K-g}) because it was used in the estimation procedure. 

The point pattern in Figure~\ref{galaxies_plot_b} was obtained by an independent thinning procedure of the point pattern in Figure~\ref{galaxies_plot_a}, with retention probability $\lambda_{\mathrm{min}}/\lambda(u)$ at location $u\in W$, where $\lambda(u)$ is given by \eqref{intensity} and $\lambda_{\mathrm{min}}:=\inf_{u\in \mathbb S^2}\lambda(u)$. 
If we imagine a realization of $\mathbf X_{\mathrm{Thom}}$ on $\mathbb S^2$ could had been observed so that
the independent thinning procedure could had taken place on the whole sphere, then the corresponding thinned Thomas process, $\mathbf X_{\mathrm{thin Thom}}$, is isotropic. Thus the commonly used $F,G,J$-functions would apply: Briefly, for a given isotropic point process $\mathbf X$ on $\mathbb S^d$, an arbitrary chosen location $u\in\mathbb S^d$, and $r\in[0,\pi]$, by definition $F(r)$ is the probability that $\mathbf X$ has a point in a cap $C(u,r)$ with center $u$ and geodesic distance $r$ from $u$ to the boundary of the cap; $G(r)$ is the conditional probability that $\mathbf X\setminus\{u\}$ has a point in $C(u,r)$ given that $u\in\mathbf X$; and $J(r)=(1-G(r))/(1-F(r))$ when $F(r)<1$. For the empirical estimates of the $F,G,J$-functions, as we only observe points within the subset $W\subset\mathbb S^2$ specified in Section~\ref{s:data}, edge correction factors will be needed, where we choose to use the border-corrected estimates from \citet{lawrence2016point}. 

Empirical estimates $\widehat F, \widehat G,\widehat J$ can then be used as test functions for the global rank envelope test, which is supplied with graphical representations of global envelopes for $\widehat F,\widehat G,\widehat J$ as described in \citet{Myllymaky2017}.  As we combined all three test functions as discussed in \citet{Mrkvicka2016} and \citet{Mrkvicka2017}, we followed their recommendation of using  $3\times 2499=7497$ simulations of $\mathbf X_{\mathrm{thin Thom}}$ for the calculation of $p$-values and envelopes. Briefly, under the claimed model, with probability 95\% we expect $\widehat F, \widehat G,\widehat J$ (the solid lines in Figure~\ref{fig:thomas_sub1}--\ref{fig:thomas_sub3}) to be within the envelope (either the dotted or dashed lines in Figure~\ref{fig:thomas_sub1}--\ref{fig:thomas_sub3} depending on the choice of $[a,b]$ as detailed below), and the envelope corresponds to a so-called global rank envelope test at level 5\%. 
In Table~1, the limits of the $p$-intervals correspond to liberal and conservative versions of the global rank envelope test \citep{Myllymaky2017}. Table~1 
and Figure~\ref{fig:thomas_sub1}--\ref{fig:thomas_sub3} clearly show that the estimated Thomas process is not providing a satisfactory fit no matter if in the contrast \eqref{e:contrast}  the long integration interval $[a,b]=[0,1.396]$ from \citet{lawrence2016point} or the shorter interval $[a,b]=[0,0.175]$ (corresponding to 0-10 degrees) is used for para\-meter estimation. 
 When $[a,b]=[0,0.175]$, larger estimates $\hat{\kappa} = 6.67$ and $\hat{\xi} = 353.94$ are obtained as compared to $\hat{\kappa} = 5.64$ and $\hat{\xi} = 266.6$ from \citet{lawrence2016point}. Note that the $p$-values and envelopes are not much affected by the choice of integration interval in the minimum contrast estimation procedure, cf.\ Table~1 and Figure~\ref{fig:thomas_sub1}--\ref{fig:thomas_sub3}. The envelopes indicate that at short inter-point distances $r$, there is more aggregation in the data than expected under the two fitted Thomas processes.

\begin{table}[!ht]
\centering
\begin{tabular}{|l||r|r|}
\hline\
&$[a,b]=[0,0.175]$ &$[a,b]=[0,1.396]$ \\
\hline\
Thomas process & 0.01\% - 1.05\% & 0.01\% - 1.28\%\\
\hline\
LGCP & 24.02\% - 24.09\% & 0.01\% - 1.23\%\\
\hline
\end{tabular}
\caption{Intervals for $p$-values obtained from the global envelope test based on combining the $F,G,J$-functions and using either a short or long integration interval $[a,b]$ when calculating the contrast used for parameter estimation in the Thomas process or the LGCP.}
\end{table}

\begin{figure}[!ht]
\begin{subfigure}{0.33\linewidth}
\centering
\includegraphics[scale=.30]{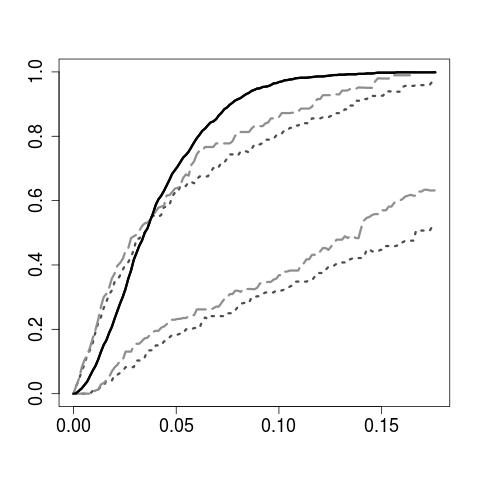}
\caption{$F$-function (Thomas).}
\label{fig:thomas_sub1}
\end{subfigure}%
\begin{subfigure}{0.33\linewidth}
\centering
\includegraphics[scale=.30]{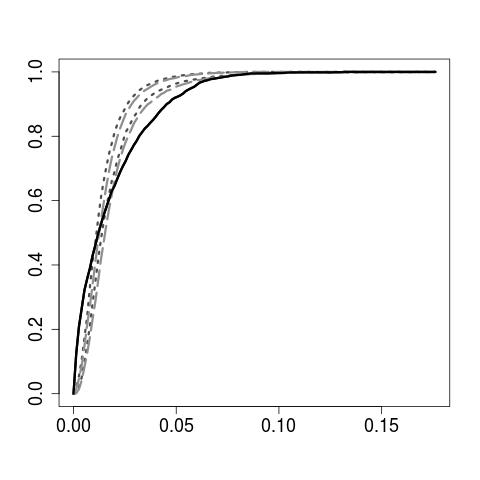}
\caption{$G$-function  (Thomas).}
\label{fig:thomas_sub2}
\end{subfigure}
\begin{subfigure}{0.33\linewidth}
\centering
\includegraphics[scale=.30]{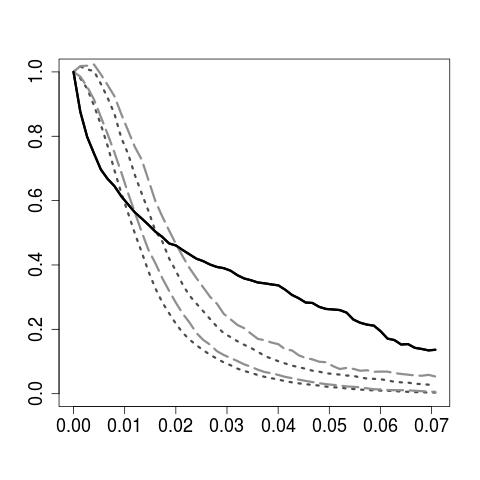}
\caption{$J$-function  (Thomas).}
\label{fig:thomas_sub3}
\end{subfigure}\\[1ex]
\begin{subfigure}{0.33\linewidth}
\centering
\includegraphics[scale=.30]{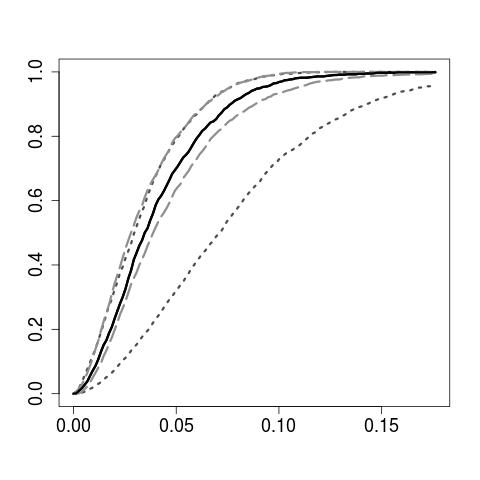}
\caption{$F$-function (LGCP).}
\label{fig:LGCP_sub1}
\end{subfigure}%
\begin{subfigure}{0.33\linewidth}
\centering
\includegraphics[scale=.30]{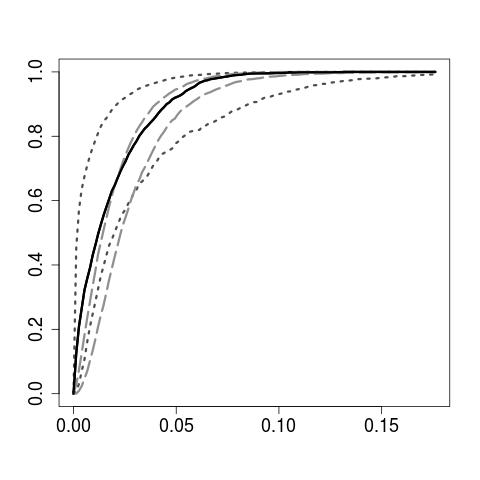}
\caption{$G$-function (LGCP).}
\label{fig:LGCP_sub2}
\end{subfigure}
\begin{subfigure}{0.33\linewidth}
\centering
\includegraphics[scale=.30]{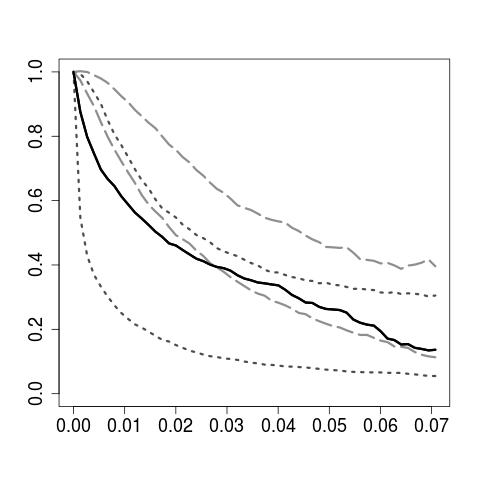}
\caption{$J$-function (LGCP).}
\label{fig:LGCP_sub3}
\end{subfigure}\caption{Empirical functional summary statistics and 95\% global envelopes under (a)--(c) fitted Thomas process and (d)--(f) fitted LGCPs for the sky positions of galaxies after thinning. The solid lines show $\widehat F(r),\widehat G(r),\widehat J(r)$ versus distance $r$ (measured in radians). The dotted and dashed lines limit the global envelopes and correspond to the short and long integration intervals $[a,b]=[0,0.175]$ and $[a,b]=[0,1.396]$, respectively, used in the minimum contrast estimation procedure.}
\label{fig:thomas}
\end{figure}

We also fitted an inhomogeneous LGCP, $\mathbf X_{\mathrm{LGCP}}$, still with intensity function given by \eqref{intensity} and with multiquadric covariance function as in \eqref{Multiquadric}. The same minimum contrast procedure as above was used except of course that in the contrast given by \eqref{e:contrast}, the theoretical $K$-function was obtained by combining  \eqref{e:K-g}, \eqref{e:LGCP-intensities}, and \eqref{Multiquadric}, where numerical calculation of the integral in \eqref{e:K-g} was used. The estimates are $(\hat{\sigma}^2,\hat{\delta},\hat{\tau})  = (4.50,0.99,0.25)$ if $[a,b]=[0,0.175]$ is the integration interval, and  $(\hat{\sigma}^2,\hat{\delta},\hat{\tau}) = (1.30,0.87,2.03)$ 
if $[a,b]=[0,1.396]$. For each choice of integration interval,
Figure~\ref{estimated_cova} shows the estimated log pair correlation function, that is, the
estimated covariance function of the underlying GRF, cf.\ \eqref{e:LGCP-intensities}. Comparing the two 
 pair correlation functions, the one based on the short integration interval is much larger for 
very short inter-point distances $r$, then rather similar to the other at a short interval of $r$-values, and afterwards again larger, so the fitted LGCP using the short integration interval is more aggregated than the other case.   
Further, Figure~\ref{K_function} shows the empirical $K$-function and the fitted $K$-functions for both the Thomas process and the LGCP when using the different integration intervals (for ease of comparison, we have subtracted $K_{\mathrm{Pois}}$, the theoretical $K$-function for a Poisson process, cf.\ \eqref{e:KPois}).  The fitted $K$-functions are far away from the empirical $K$-function for large values of $r$. However, having a good agreement for small and modest values of $r$ is more important, because the variance of the empirical $K$-function seems to be an increasing function of $r$ and the interpretation of the $K$-function becomes harder for large $r$-values. For small and modest $r$-values, using the LGCP model and the short integration interval provides the best agreement between the empirical $K$-function and the theoretical $K$-function, even when $r$ is outside the short integration interval.   

\begin{figure} 
\centering
{\includegraphics[scale=0.5]{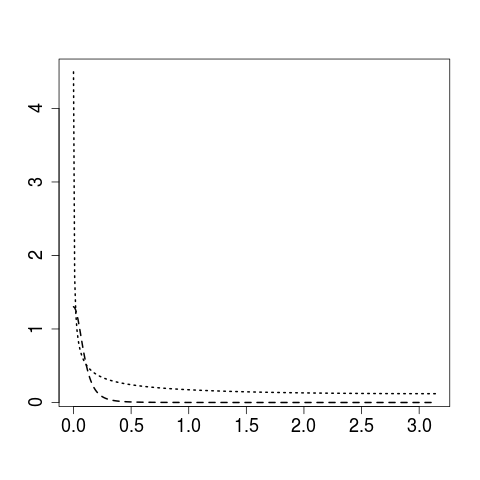}}
\caption{Plot of estimated pair correlation functions $g(r)$ (on a logarithmic scale) versus $r\in[0,\pi]$ for the LGCP when different integration intervals $[a,b]$ were used in the minumum contrast estimation procedure (dotted line: $[a,b]=[0,0.175]$; dashed line: $[a,b]=[0,1.396]$).}\label{estimated_cova}
\end{figure}

\begin{figure} 
\begin{subfigure}{.5\linewidth}
\centering
\includegraphics[scale=0.4]{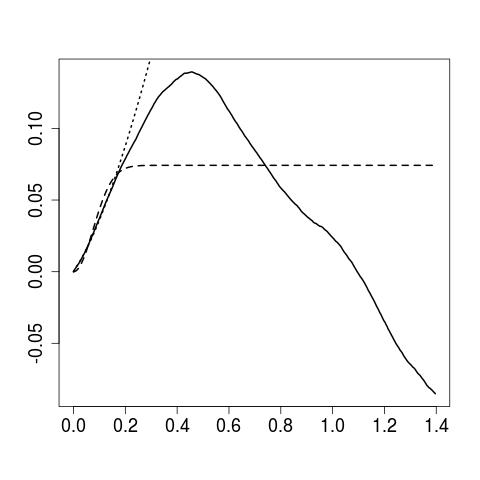}
\caption{Integration interval [0,0.175].}\label{small_interval}
\end{subfigure}
\begin{subfigure}{.5\linewidth}
\centering
\includegraphics[scale=0.4]{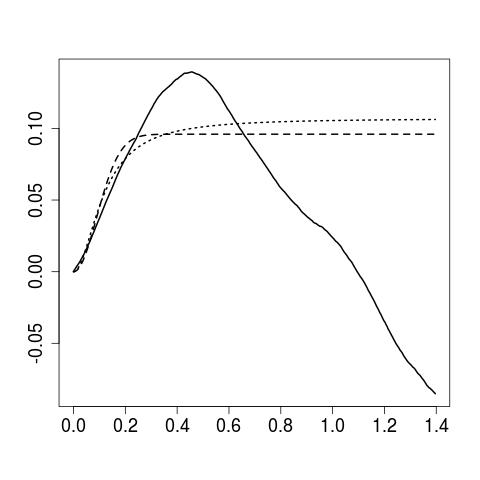}
\caption{Integration interval [0,1.396].}\label{big_interval}
\end{subfigure}
\caption{Fitted $K$-functions minus the theoretical Poisson $K$-function versus distance $r$ for the sky positions of galaxies when using different integration intervals in the minumum contrast estimation procedure. The solid lines correspond to the empirical functional summary statistics, and the dashed and dotted lines correspond to the theoretical functional summary statistics under the fitted Thomas processes  and LGCPs, respectively.}\label{K_function}
\end{figure}

Table~1 and 
Figure~\ref{fig:LGCP_sub1}--\ref{fig:LGCP_sub3} show the results for the fitted LGCPs when making a model checking along similar lines as for the fitted Thomas processes (i.e., based on a thinned LGCP 
and considering the empirical $F, G,J$-functions together with global envelopes). 
Table~1 shows that the fitted LGCP based on the long integration interval gives an interval of similar low $p$-values as for the fitted Thomas process in \citet{lawrence2016point}, and Figure~\ref{fig:LGCP_sub2}--\ref{fig:LGCP_sub3} indicate that the data is more aggregated than expected under this fitted LGCP. Finally, the fitted LGCP based on the short integration interval gives $p$-values of about 24\%, and the empirical curves of the functional summary statistics appear in the center of the envelopes, cf.\ Figure~\ref{fig:LGCP_sub1}--\ref{fig:LGCP_sub3}. 

\begin{figure}[!ht]
\begin{subfigure}{0.33\linewidth}
\centering
\includegraphics[scale=.30]{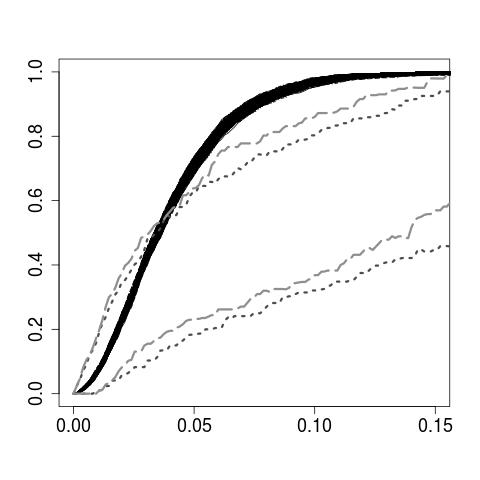}
\caption{$F$-function (Thomas).}
\label{fig:thomas_sub_rep1}
\end{subfigure}%
\begin{subfigure}{0.33\linewidth}
\centering
\includegraphics[scale=.30]{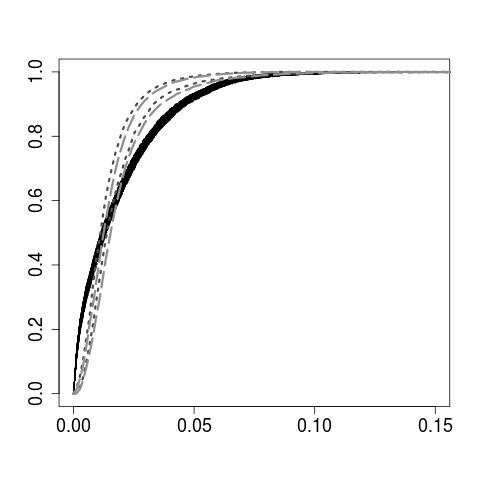}
\caption{$G$-function  (Thomas).}
\label{fig:thomas_sub_rep2}
\end{subfigure}
\begin{subfigure}{0.33\linewidth}
\centering
\includegraphics[scale=.30]{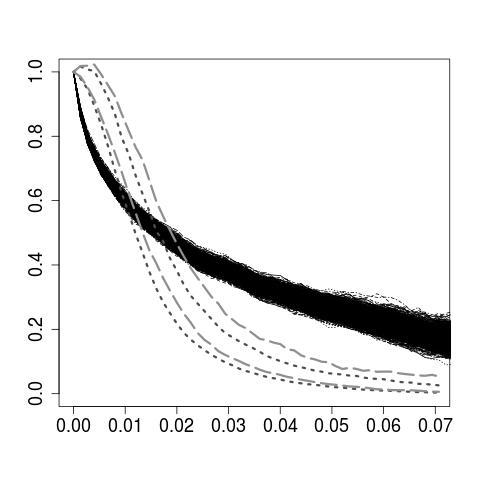}
\caption{$J$-function  (Thomas).}
\label{fig:thomas_sub_rep3}
\end{subfigure}\\[1ex]
\begin{subfigure}{0.33\linewidth}
\centering
\includegraphics[scale=.30]{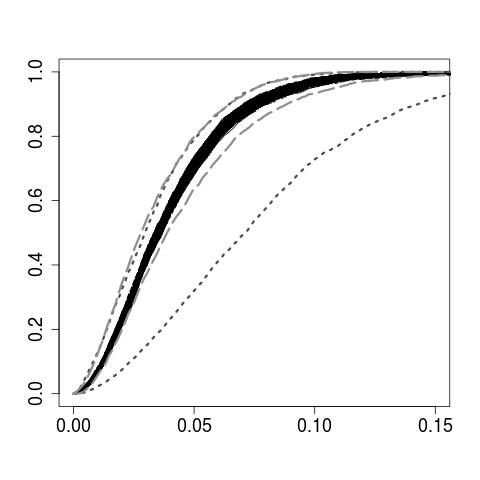}
\caption{$F$-function (LGCP).}
\label{fig:LGCP_sub_rep1}
\end{subfigure}%
\begin{subfigure}{0.33\linewidth}
\centering
\includegraphics[scale=.30]{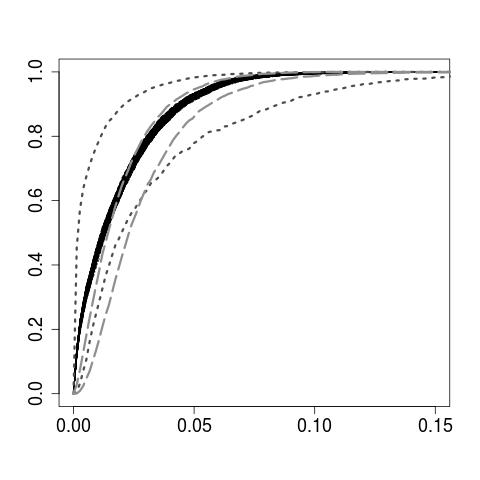}
\caption{$G$-function (LGCP).}
\label{fig:LGCP_sub_rep2}
\end{subfigure}
\begin{subfigure}{0.33\linewidth}
\centering
\includegraphics[scale=.30]{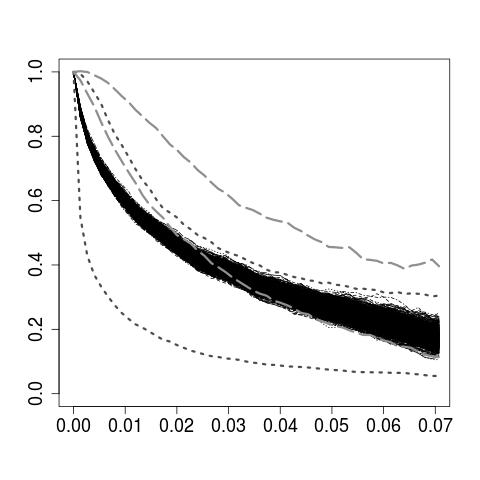}
\caption{$J$-function (LGCP).}
\label{fig:LGCP_sub_rep3}
\end{subfigure}\caption{Empirical functional summary statistics and 95\% global envelopes under (a)--(c) fitted Thomas process and (d)--(f) fitted LGCPs for the sky positions of galaxies after thinning. The 1000 solid lines show $\widehat F(r),\widehat G(r),\widehat J(r)$ versus distance $r$ (measured in radians) when the independent thinning procedure is repeated 1000 times. The dotted and dashed lines limit the global envelopes and correspond to the short and long integration intervals $[a,b]=[0,0.175]$ and $[a,b]=[0,1.396]$, respectively, used in the minimum contrast estimation procedure.}
\label{fig:thomas_rep}
\end{figure}

\begin{figure}[!ht]
\centering
\includegraphics[scale=.50]{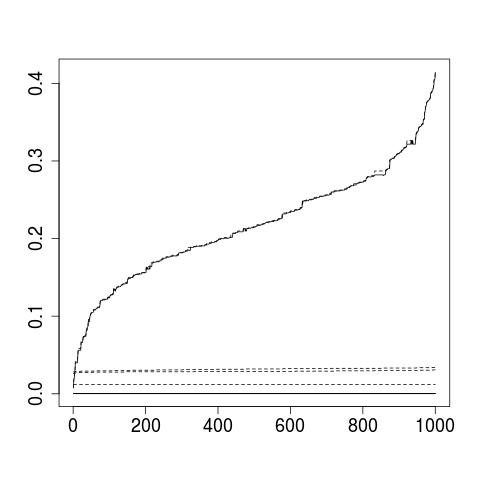}
\caption{Intervals for $p$-values obtained from the global envelope test based on combining the $F,G,J$-functions when repeating the independent thinning procedure 1000 times. The three lower solid lines are very close and therefore appear as one thick solid line in the plot. Each pair of dashed and solid lines from below to the top corresponds to conservative and liberal $p$-values for the fitted Thomas process, using first the long and second the short integration interval, and for the fitted LGCP, using first the long and second the short integration interval.}
\label{fig:intervals}
\end{figure}

Finally, we have studied how sensible the results will be when using  the independent thinning procedure. Figure~\ref{fig:thomas_rep} is similar to Figure~\ref{fig:thomas} except that the 1000 solid lines show the empirical estimates $\widehat F,\widehat G,\widehat J$ when we repeat the thinning procedure 1000 times. The empirical estimates do not vary much in the 1000 cases. 
Moreover, Figure~\ref{fig:intervals} is similar to Table~1 and shows the intervals for $p$-values obtained in the 1000 cases. 
The intervals for the fitted Thomas process, using the short or long integration interval, and for the fitted LGCP, using the long integration interval, are effectively not varying in the 1000 cases (the three lower solid curves correspond to liberal $p$-values and visually they appear as one curve which is close to 0: the curve in the LGCP case varies around $10^{-4}$, and the two curves in the Thomas process case vary around $4\times10^{-4}$). The intervals for the fitted LGCP, using the short integration interval, are much shorter than for the three other fitted models, and they vary from 0.75\% - 0.77\% to 41.44\% - 41.48\%, with only 12 out of the 1000 intervals below the 5\% level. Thus, the results are not so sensitive to the independent thinning procedure and the conclusion, namely that the fitted LGCP, using the short integration interval, is fitting well whilst the other fitted models do not, is almost the same in the 1000 cases.

\section{Further results and concluding remarks}\label{s:final}

\subsection{Existence}\label{s:final1}

As noticed in Section~\ref{s:DefExLGCP}, under mild conditions for the mean function $\mu$, almost sure locally sample H\"{o}lder continuity of the mean-zero GRF $\mathbf Y_0=\mathbf Y-\mu$ is a sufficient condition for  
establishing the existence of a LGCP driven by $\exp(\mathbf Y)$. In turn, almost sure locally sample H\"{o}lder continuity of $\mathbf Y_0$ is implied by the condition \eqref{condition_hc} in Proposition~\ref{p:1}, and  
Proposition~\ref{p:3} states that the multiquadric covariance function is satisfying this condition. As shown in the following proposition, \eqref{condition_hc} is satisfied for any of the parametric classes of isotropic covariance functions on $\mathbb S^d$ given by \citet{gneiting2013strictly}, see Table~\ref{tab:covariances}. In brief, these are the 
commonly used isotropic covariance functions which are expressible on closed form.  

\begin{table}[H]
    \centering
    \begin{tabular}{l||l|l}
        Model & Correlation function $c_0(r)$ & Parameter range \\
        \hline
        \hline
        Powered exponential & $\exp\left(-r^\alpha/\phi\right)$ & $\alpha\in(0,1]$, $\phi>0$\\
        Mat{\'e}rn & $\frac{2}{\Gamma(\nu)}\left( \frac{r}{2\phi} \right)^{\nu} K_{\nu}\left(\frac{r}{\phi}\right)$ & $0< \nu \leq \frac{1}{2}$, $\phi>0$\\
        Generalized Cauchy & $(1 + (\frac{r}{\phi})^{\alpha})^{-\tau/\alpha}$ & $\phi > 0$, $\alpha \in (0,1]$, $\tau > 0$\\
        Dagum & $1 - ((\frac{r}{\phi})^{\tau}/(1+(\frac{r}{\phi})^{\tau}))^{\frac{\alpha}{\tau}}$ & $\phi > 0$, $\tau \in (0,1]$ \\
        multiquadric & $\left(\frac{(1-\delta)^2}{1+\delta^2-2\delta\cos r}\right)^\tau$ & $\delta\in(0,1)$, $\tau>0$\\
        Sine power & $1- \sin(r/2)^{\alpha}$ & $\alpha \in (0,2)$ \\
        Spherical & $(1 + \frac{1}{2}\frac{r}{\phi})(1-\frac{r}{\phi})_{+}^{2}$ & $\phi > 0$  \\
        Askey & $(1-\frac{r}{\phi})_{+}^{\tau}$ & $\phi > 0$, $\tau \geq 2$ \\
        $C^{2}$-Wendland & $(1 + \tau\frac{r}{\phi})(1-\frac{r}{\phi})_{+}^{\tau}$ & $\phi \in (0,\pi]$, $\tau \geq 4$ \\
        $C^{4}$-Wendland & $(1 + \tau\frac{r}{\phi} + \frac{\tau^{2}-1}{3}\frac{r^{2}}{\phi^{2}})(1-\frac{r}{\phi})_{+}^{\tau}$ & $\phi \in (0,\pi]$, $\tau \geq 6$ \\
    \end{tabular}
    \caption{Parametric models for an isotropic correlation function $c_0(r)$ on $\mathbb{S}^{d}$, where $0\le r\le\pi$, $\Gamma$ is the gamma function, $K_{\nu}$ is the modified Bessel function of the second kind, and $t_{+} := \max\{t,0\}$ for $t\in\mathbb R$. For the powered exponential, Mat{\'e}rn,  generalized Cauchy, Dagum, multiquadric, and sine power models, $d\in \{1,2,\ldots\}$, whilst for the spherical, Askey, $C^2$-Wendland, and $C^4$-Wendland models, $d\in\{1,2,3\}$. For each model, the specified parameter range ensures that $c_0(r)$ is well-defined, cf.\ \citet{gneiting2013strictly}, and hence for any $\sigma^2>0$, $c(r):=\sigma^2c_0(r)$ is an isotropic covariance function.} 
    \label{tab:covariances}
\end{table}

\begin{prop}\label{p:gneitingclasses}
Any isotropic covariance function $c$ as given by Table~\ref{tab:covariances} is satisfying \eqref{condition_hc} for some $s \in (0,1],$ $\ell \in (0,1),$ and $m>0$.
\end{prop}
\begin{proof}
We have already verified this for the multiquadric model, cf.\ Proposition~\ref{p:3}. For any of
models in Table~\ref{tab:covariances}, the variance $c(0)$ is strictly positive, so
dividing the left and right side in \eqref{condition_hc} by $c(0)$, we can without loss of generality assume the covariance function $c(r)$ to be a correlation function, i.e.\ $c=c_0$, cf.\ the caption to Table~\ref{tab:covariances}. 
Hence we need only to show the existence of numbers $s\in(0,1],$ $\ell\in(0,1)$, and $m>0$ so that the condition
\begin{equation}\label{e:aaaa}
c_0(0)-c_0(r)\le m r^{\ell/2}\qquad\mbox{whenever }r<s
\end{equation}
is satisfied, where $c_0(0)=1$. 

Consider the powered exponential model. 
If $0<\alpha<1$ and $r\le 1$, then $r^\alpha\le r^{\alpha/2}$, and so we obtain \eqref{e:aaaa} because
\[1-c_0(r)\le r^\alpha/\phi\le r^{\alpha/2}/\phi,\]
where the first inequality follows from $\exp(-x)\ge 1-x$. Further, 
the case $\alpha=1$ is the special case of the Mat{\'e}rn model with shape parameter $\nu=1/2$, which is considered below.
For the remaining eight models in Table~\ref{tab:covariances}, we establish a limit 
\begin{equation}\label{limit_prop}
    \lim_{r \downarrow 0} \frac{c_0(0) - c_0(r)}{r^{A}} = B,
\end{equation}
for some $A>0$ and $B>0$, where $(A,B)$ depends on the specific model. 
The limit \eqref{limit_prop} implies that \eqref{e:aaaa} holds for some $s\in(0,1]$, $\ell=\min\{A/2,1-\epsilon\}$ with $0<\epsilon<1$, and $m>B$. Note that $0<\ell<1$.

For example, consider the Mat{\'e}rn model. 
By DLMF (2018, Equation 10.27.4)\nocite{NIST:DLMF}, when $0<\nu\le\frac12$,
$$K_{\nu}(r) = \frac{\pi}{2}\frac{\left(\frac{r}{2}\right)^{-\nu}E_{-\nu}(r) - \left(\frac{r}{2}\right)^{\nu}E_{\nu}(r)}{\sin(\nu\pi)},$$
where
$$E_{\nu}(r) = \sum_{n = 0}^{\infty} \frac{r^{2n}}{n!\Gamma(n + \nu + 1)4^{n}}.$$
Thus, defining $A_{\nu} = \pi/(\sin(\pi\nu)\Gamma(\nu))$, the expression of the Mat{\'e}rn correlation function in Table~\ref{tab:covariances} can be rewritten as
$$c_0(r) = A_{\nu}\left(E_{-\nu}\left(\frac{r}{\phi}\right) - \left(\frac{r}{2\phi}\right)^{2\nu}E_{\nu}\left(\frac{r}{\phi}\right)\right).$$
Then 
\[ \lim_{r \downarrow 0}\frac{c_0(0)-c_0(r)}{r^{2\nu}}=-A_\nu \lim_{r \downarrow 0}\sum_{n=1}^\infty\frac{(r/\phi)^{2n}r^{-2\nu}}{n!\Gamma(n + \nu + 1)4^{n}} + A_\nu(2\phi)^{-2\nu}\lim_{r \downarrow 0}E_\nu(r/\phi)=A_\nu/((2\phi)^{2\nu}\Gamma(\nu+1)),\]
and so \eqref{limit_prop} holds, with $A=2\nu$ and $B=A_\nu/((2\phi)^{2\nu}\Gamma(\nu+1))$. 
 
 For the remaining seven models, the values of $A$ and $B$ in \eqref{limit_prop} are straightforwardly derived, using L'Hospital's rule when considering the generalized Cauchy, Askey, $C^2$-Wendland, and $C^4$-Wendland models.
  The values are given in Table~\ref{tab:covariances}.

\begin{table}[H]
    \centering
    \begin{tabular}{l||l|l}
        Model & $A$ & $B$ \\
        \hline
        \hline
        Generalized Cauchy & $\alpha$ & $\tau/(\alpha \phi^{\alpha})$ \\
        Dagum & $\alpha$ & $\phi^{-\alpha}$ \\
        Sine power & $\alpha$ & $2^{-\alpha}$ \\
        Spherical & 1 & $\frac{3}{2\phi}$ \\
        Askey & 1 & $\tau/\phi$ \\
        $C^{2}$-Wendland & 2 & $\tau(\tau+1)\phi^{-2}/2$\\
        $C^{4}$-Wendland & 2 & $(\tau+1)(\tau+2)\phi^{-2}/3$ \\
    \end{tabular}
    \caption{List of $A$ and $B$ values for the last seven models considered in the proof of Proposition~\ref{p:gneitingclasses}.} 
    \label{tab:covariances2}
\end{table}

\end{proof} 

\subsection{Moment properties, Palm distribution, and statistical inference}\label{s:final2}

As demonstrated, a LGCP on the sphere possesses useful theoretical properties, in particular moment properties as provided by Proposition~\ref{p:2}.
We exploited these expressions for the intensity and pair correlation function when dealing with the inhomogeneous $K$-function, which concerns the second moment properties of a second order intensity reweighted homogeneous point process. 
 
Proposition~\ref{p:2} extends as follows. For a general point process on $\mathbb S^d$, the $n$-th order pair correlation function $g(u_1,\ldots,u_n)$ is defined for integers $n\ge2$ and multiple disjoint regions $U_1\ldots,U_n\subseteq\mathbb S^d$ by 
\[\mathrm E\left[N\left(U_1\right)\cdots N\left(U_n\right)\right]=\int_{U_1}\cdots\int_{U_n}\lambda(u_1)\cdots\lambda(u_n)\,g(u_1,\ldots,u_n)\,\mathrm du_1\cdots\,\mathrm du_n\]
provided this multiple integral is well-defined and finite. The following proposition follows immediately as in \citet[][Theorem 1]{MW03}.  
\begin{prop}\label{p:5}
For any integer $n\ge2$, 
a LGCP on $\mathbb S^d$ has $n$-th order pair correlation function  given by 
\begin{equation}\label{e:gn} g(u_1,\ldots,u_n)=\exp\left(\sum_{1\le i<j\le n}c(u_i,u_j)\right),\end{equation}
where $c$ is the covariance function of the underlying GRF. 
\end{prop}

  Higher-order pair correlation functions as given by \eqref{e:gn} may be used for constructing further functional summaries e.g.\ along similar lines as the third-order characteristic studied in \citet{Moller1998} and the  inhomogeneous $J$-function studied in 
  \citet[][Section 5.3]{Cronie:Lieshout:15}.
  
It would also be interesting to exploit the following result for the reduced Palm distribution of a LGCP
 on the sphere 
as discussed in \citet{Coeur17} in the case of LGCPs on $\mathbb R^d$. Intuitively,
the reduced Palm distribution of a point process $\mathbf X$ on $\mathbb S^d$ at a given point $u\in\mathbb S^d$ corresponds to  
the distribution of $\mathbf X\setminus\{u\}$ conditional on that $u\in\mathbf X$; see \citet{lawrence2016point} and \citet{Moller_summary_2017}. Along similar lines as in \citet[][Theorem~1]{Coeur17}, we obtain immediately the following result.
\begin{prop}\label{p:4}
Consider a LGCP $\mathbf X$ whose underlying GRF has mean and covariance functions $\mu$ and $c$.
For any $u\in\mathbb S^d$, the reduced Palm distribution of $\mathbf X$ at $u$ is a LGCP, where the underlying GRF has unchanged covariance function $c$ but mean function $m_u(v)=m(v)+c(u,v)$ for $v\in\mathbb S^d$.
\end{prop}

Bayesian analysis of inhomogeneous LGCPs on the sphere has previously been considered in \citet{Simpson.et.al} using the R-INLA approach: Their covariance model for the underlying GRF 
is an extension of the Mat{\'e}rn covariance function
 defined as the covariance function at any fixed time to the solution to a stochastic partial differential equation which is stationary in time and isotropic on the sphere. Using the R-INLA approach, a finite approximation of their LGCP based on a triangulation of the sphere is used, which for computational efficiency requires $\alpha:=\nu+d/2$ to be an integer, where $\nu>0$ is the shape parameter (the expression $\alpha=\nu-d/2$ in their paper is a typo). 
\citet{Simpson.et.al} considered an inhomogeneous LGCP defined only on the world’s oceans, with $d=2$ and $\alpha=2$, so $\nu=1$. 
 In the original Mat{\'e}rn model \citep[][Table 1]{gneiting2013strictly}, $0<\nu\le0.5$ and the covariance function has a nice expression in term of a Bessel function, but otherwise the covariance function can only be expressed by an infinite series in terms of spherical harmonics \citep[][Section 6.4]{SimpsonThesis}. In contrast, we have focused on covariance functions $c$ which are expressible on closed form so that we can easily work with the pair correlation function $g=\exp(c)$. Indeed it could be interesting to apply
  the R-INLA approach as well as other likelihood based methods of inference \citep{moeller:waage17}, though they are certainly much more complicated than the simple inference procedure considered in the present paper. 
 
For comparison with the analysis of the sky positions of galaxies in \citet{lawrence2016point}, we used a minimum contrast estimation procedure, but other simple and fast methods such as composite likelihood \citep[][and the references therein]{guan:06,moeller:waage17} could have been used as well. 
It is well-known that such estimation procedures can be sensitive to the choice of user-specified parameters. We demonstrated this for the choice of integration interval in the contrast, where using a short interval and an inhomogeneous LGCP provided a satisfactory fit, in contrast to using an inhomogeneous Thomas process or a long interval. It could be interesting to use more advanced estimation procedures such as maximum likelihood and Bayesian inference. This will involve a time-consuming missing data simulation-based approach  \citep[][and the references therein]{MW04,moeller:waage17}.

\subsection{Modelling the sky positions of galaxies}  

The Thomas process is a mechanistic model since it has an interpretation as a cluster point process \citep{lawrence2016point}. The original Thomas process in $\mathbb R^3$ (i.e., using a 3-dimensional isotropic zero-mean normal distribution as the density for a point relative to its cluster centre)  may perhaps appear natural for positions of galaxies in the 3-dimensional space -- but we question if the Thomas process from \citet{lawrence2016point} is a natural model for the sky positions because these points are obtained by projecting clusters of galaxies in space to a sphere which may not produce a clear clustering because of overlap. Rather, we think both the inhomogeneous Thomas and the inhomogeneous LGCP should be viewed as empirical models for the data. Moreover, it could be investigated if the Thomas process replaced by another type of Neyman-Scott process \citep{neyman1958statistical,neyman1972processes} or a (generalised) shot noise Cox process \citep{moeller:03,moeller:torrisi:05} would provide an adequate fit. 

As a specific model, we only considered the multiquadric covariance function for the underlying GRF of the LGCP. A more flexible model could be the spectral model studied in \citet{Moller2018}, where a parametric model for the eigenvalues of the spectral representation of an isotropic covariance function $c(r)$ ($r\ge0$) in terms of spherical harmonics is used (incidentally, the eigenvalues are also known for the multiquadric covariance function if $\tau=1/2$). The spectral representation allows a Karhunen-Lo\`{e}ve representation which could be used for simulation. However, for the data analysed in this paper, we found it easier and faster just to approximate the GRF $\mathbf Y=\{\mathbf Y(u):u\in\mathbb S^2\}$, using a finite grid $I\subset\mathbb S^2$ so that each $\mathbf Y(u)$ is replaced by $\mathbf Y(v)$ if $v$ is the nearest grid point to $u$, and using the singular value decomposition when simulating the finite random field $\{\mathbf Y(v):v\in I\}$. When using spherical angles $(\theta,\psi)$ as in \eqref{intensity}, a regular grid over $[0,\pi]\times[0,2\pi)$ can not be recommended, since the density of grid points will large close to the poles and small close to equator.  Using a regular grid $I\subset\mathbb S^2$ avoids this problem, nonetheless, there are only five regular grids on the sphere \citep{coxeter1973regular}. We used a nearly-regular grid consisting of 4098 points on the sphere \citep[][and the references therein]{Szalay:2005} and computed using the R package \texttt{mvmesh} \citep{nolan:2016}.

\subsubsection*{Acknowledgements}
Supported by The Danish Council for Independent Research $|$ Natural Sciences, grant DFF – 7014-00074
"Statistics for point processes in space and beyond", and by the "Centre for Stochastic Geometry and 
Advanced Bioimaging", funded by grant 8721 from the Villum Foundation.

\section*{References}
\bibliographystyle{abbrvnat}
\bibliography{bibliography}

\begin{thebibliography}{33}
\providecommand{\natexlab}[1]{#1}
\providecommand{\url}[1]{\texttt{#1}}
\expandafter\ifx\csname urlstyle\endcsname\relax
  \providecommand{\doi}[1]{doi: #1}\else
  \providecommand{\doi}{doi: \begingroup \urlstyle{rm}\Url}\fi

\bibitem[Baddeley et~al.(2000)Baddeley, M{\o}ller, and
  Waagepetersen]{Baddeley2000}
A.~J. Baddeley, J.~M{\o}ller, and R.~Waagepetersen.
\newblock Non- and semi-parametric estimation of interaction in inhomogeneous
  point patterns.
\newblock \emph{Statistica Neerlandica}, 54:\penalty0 329--350, 2000.

\bibitem[Coeurjolly et~al.(2017)Coeurjolly, M{\o}ller, and
  Waagepetersen]{Coeur17}
J.-F. Coeurjolly, J.~M{\o}ller, and R.~Waagepetersen.
\newblock Palm distributions for log {G}aussian {C}ox processes.
\newblock \emph{Scandinavian Journal of Statistics}, 44:\penalty0 192--203,
  2017.

\bibitem[Coxeter(1973)]{coxeter1973regular}
H.~S.~M. Coxeter.
\newblock \emph{Regular {P}olytopes}.
\newblock Methuen, London, 1973.

\bibitem[Cronie and van Lieshout(2015)]{Cronie:Lieshout:15}
O.~Cronie and M.~N.~M. van Lieshout.
\newblock A {$J$}-function for inhomogeneous spatio-temporal point processes.
\newblock \emph{Scandinavian Journal of Statistics}, 42:\penalty0 562--579,
  2015.

\bibitem[Diggle(2013)]{diggle2013statistical}
P.~J. Diggle.
\newblock \emph{Statistical {A}nalysis of {S}patial and {S}patio-temporal
  {P}oint {P}atterns}.
\newblock CRC Press, Boca Raton, 2013.

\bibitem[Diggle and Gratton(1984)]{Diggle1983}
P.~J. Diggle and R.~J. Gratton.
\newblock Monte {C}arlo methods of inference for implicit statistical models.
\newblock \emph{Journal of the Royal Statistical Society: Series B (Statistical
  Methodology)}, 46:\penalty0 193--227, 1984.

\bibitem[{\relax DLMF}()]{NIST:DLMF}
{\relax DLMF}.
\newblock {\it NIST Digital Library of Mathematical Functions}.
\newblock http://dlmf.nist.gov/, Release 1.0.18 of 2018-03-27.
\newblock URL \url{http://dlmf.nist.gov/}.
\newblock F.~W.~J. Olver, A.~B. {Olde Daalhuis}, D.~W. Lozier, B.~I. Schneider,
  R.~F. Boisvert, C.~W. Clark, B.~R. Miller and B.~V. Saunders, eds.

\bibitem[Gneiting(2013)]{gneiting2013strictly}
T.~Gneiting.
\newblock Strictly and non-strictly positive definite functions on spheres.
\newblock \emph{Bernoulli}, 19:\penalty0 1327--1349, 2013.

\bibitem[Guan(2006)]{guan:06}
Y.~Guan.
\newblock A composite likelihood approach in fitting spatial point process
  models.
\newblock \emph{Journal of the American Statistical Association}, 101:\penalty0
  1502--1512, 2006.

\bibitem[Lang et~al.(2016)Lang, Potthoff, Schlather, and
  Schwab]{lang2016continuity}
A.~Lang, J.~Potthoff, M.~Schlather, and D.~Schwab.
\newblock Continuity of random fields on {R}iemannian manifolds.
\newblock \emph{\rm{Manuscript available at arXiv:} 1607.05859}, 2016.

\bibitem[Lawrence et~al.(2016)Lawrence, Baddeley, Milne, and
  Nair]{lawrence2016point}
T.~Lawrence, A.~Baddeley, R.~K. Milne, and G.~Nair.
\newblock Point pattern analysis on a region of a sphere.
\newblock \emph{Stat}, 5:\penalty0 144--157, 2016.

\bibitem[M{{\o}}ller(2003)]{moeller:03}
J.~M{{\o}}ller.
\newblock Shot noise {C}ox processes.
\newblock \emph{Advances in Applied Probability}, 35:\penalty0 614--640, 2003.

\bibitem[M{\o}ller and Rubak(2016)]{Moller_summary_2017}
J.~M{\o}ller and E.~Rubak.
\newblock Functional summary statistics on the sphere with an application to
  determinantal point processes.
\newblock \emph{Spatial Statistics}, 18:\penalty0 4--23, 2016.

\bibitem[M{\o}ller and Torrisi(2005)]{moeller:torrisi:05}
J.~M{\o}ller and G.~Torrisi.
\newblock Generalised shot noise {C}ox processes.
\newblock \emph{Advances in Applied Probability}, 37:\penalty0 48--74, 2005.

\bibitem[M{\o}ller and Waagepetersen(2017)]{moeller:waage17}
J.~M{\o}ller and R.~Waagepetersen.
\newblock Some recent developments in statistics for spatial point patterns.
\newblock \emph{Annual Review of Statistics and Its Applications}, 4:\penalty0
  317--342, 2017.

\bibitem[M{\o}ller and Waagepetersen(2003)]{MW03}
J.~M{\o}ller and R.~P. Waagepetersen.
\newblock \emph{Spatial Statistics and Computational Methods}, chapter An
  introduction to simulation-based inference for spatial point processes, pages
  37--60.
\newblock Lecture Notes in Statistics 173, Springer-Verlag, New York, 2003.

\bibitem[M{\o}ller and Waagepetersen(2004)]{MW04}
J.~M{\o}ller and R.~P. Waagepetersen.
\newblock \emph{Statistical Inference and Simulation for Spatial Point
  Processes}.
\newblock Boca Raton, FL, Chapman \& Hall/CRC., 2004.

\bibitem[M{\o}ller et~al.(1998)M{\o}ller, Waagepetersen, and
  Syversveen]{Moller1998}
J.~M{\o}ller, R.~Waagepetersen, and A.~R. Syversveen.
\newblock Log {G}aussian {C}ox processes.
\newblock \emph{Scandinavian Journal of Statistics}, 25:\penalty0 451--482,
  1998.

\bibitem[M{\o}ller et~al.(2018)M{\o}ller, Nielsen, Porcu, and
  Rubak]{Moller2018}
J.~M{\o}ller, M.~Nielsen, E.~Porcu, and E.~Rubak.
\newblock Determinantal point process models on the sphere.
\newblock \emph{Bernoulli}, 24:\penalty0 1171--1201, 2018.

\bibitem[Mrkvi\u{c}ka et~al.(2016)Mrkvi\u{c}ka, Soubeyrand, Myllym{\"a}ki,
  Grabarnik, and Hahn]{Mrkvicka2016}
T.~Mrkvi\u{c}ka, S.~Soubeyrand, M.~Myllym{\"a}ki, P.~Grabarnik, and U.~Hahn.
\newblock Monte {C}arlo testing in spatial statistics, with applications to
  spatial residuals.
\newblock \emph{Spatial Statistics}, 18:\penalty0 40--53, 2016.

\bibitem[Mrkvi\u{c}ka et~al.(2017)Mrkvi\u{c}ka, Myllym{\"a}ki, and
  Hahn]{Mrkvicka2017}
T.~Mrkvi\u{c}ka, M.~Myllym{\"a}ki, and U.~Hahn.
\newblock Multiple {M}onte {C}arlo testing, with applications in spatial
  points.
\newblock \emph{Statistics and Computing}, 27:\penalty0 1239--1255, 2017.

\bibitem[Myllym{\"a}ki et~al.(2017)Myllym{\"a}ki, Mrkvi\u{c}ka, Grabarnik,
  Seijo, and Hahn]{Myllymaky2017}
M.~Myllym{\"a}ki, T.~Mrkvi\u{c}ka, P.~Grabarnik, H.~Seijo, and U.~Hahn.
\newblock Global envelope tests for spatial processes.
\newblock \emph{Journal of the Royal Statistical Society: Series B (Statistical
  Methodology)}, 79:\penalty0 381--404, 2017.

\bibitem[Neyman and Scott(1958)]{neyman1958statistical}
J.~Neyman and E.~L. Scott.
\newblock Statistical approach to problems of cosmology.
\newblock \emph{Journal of the Royal Statistical Society: Series B (Statistical
  Methodology)}, 20:\penalty0 1--43, 1958.

\bibitem[Neyman and Scott(1972)]{neyman1972processes}
J.~Neyman and E.~L. Scott.
\newblock \emph{{P}rocesses of {C}lustering and {A}pplications}.
\newblock John Wiley \& Sons, New York, {S}tochastic {P}rocesses in {L}ewis
  {PAW} edition, 1972.

\bibitem[Nolan(2016)]{nolan:2016}
J.~P. Nolan.
\newblock An {R} package for modeling and simulating generalized spherical and
  related distributions.
\newblock \emph{Journal of Statistical Distributions and Applications},
  3:\penalty0 3--14, 2016.

\bibitem[Raskin(1994)]{Raskin94}
R.~G. Raskin.
\newblock Spatial analysis on the sphere: a review, 1994.
\newblock NGCIA technical report. (Available from
  http://eprints.cdlib.org/uc/item/5748n2xz). [Accessed February 1, 2018].

\bibitem[Ripley(1976)]{Ripley76}
B.~D. Ripley.
\newblock The second-order analysis of spatial point processes.
\newblock \emph{Journal of Applied Probability}, 13:\penalty0 255--266, 1976.

\bibitem[Ripley(1977)]{Ripley77}
B.~D. Ripley.
\newblock Modelling spatial patterns (with discussion).
\newblock \emph{Journal of the Royal Statistical Society: Series B (Statistical
  Methodology)}, 39:\penalty0 172--212, 1977.

\bibitem[Robeson et~al.(2014)Robeson, Li, and Huang]{Robeson2014}
S.~M. Robeson, A.~Li, and C.~Huang.
\newblock Point-pattern analysis on the sphere.
\newblock \emph{Spatial Statistics}, 10:\penalty0 76--86, 2014.

\bibitem[Simpson(2009)]{SimpsonThesis}
D.~Simpson.
\newblock \emph{Krylov Subspace Methods for Approximating Functions of
  Symmetric Positive Definite Matrices with Applications to Applied Statistics
  and Anomalous Diffusion}.
\newblock PhD thesis, Queensland University of Technology, 2009.

\bibitem[Simpson et~al.(2016)Simpson, Illian, Lindgren, S{\o}rbye, and
  Rue]{Simpson.et.al}
D.~Simpson, J.~B. Illian, F.~Lindgren, S.~H. S{\o}rbye, and H.~Rue.
\newblock Going off grid: computationally efficient inference for
  log-{G}aussian {C}ox processes.
\newblock \emph{Biometrika}, 103:\penalty0 49–70, 2016.

\bibitem[Steinicke(2015)]{steinicke2015}
W.~Steinicke.
\newblock Revised new general catalogue and index catalogue, 2015.
\newblock (Available from http://www.klima-luft.de/steinicke/index.htm).
  [Accessed February 1, 2018].

\bibitem[Szalay et~al.(2007)Szalay, Gray, Fekete, Kunszt, Kukol, and
  Thakar]{Szalay:2005}
A.~S. Szalay, J.~Gray, G.~Fekete, P.~Z. Kunszt, P.~Kukol, and A.~Thakar.
\newblock Indexing the sphere with the hierarchical triangular mesh.
\newblock \emph{\rm{Manuscript available at arXiv:} 0701164}, 2007.

\end{thebibliography}

\end{document}